\documentclass[bibalpha]{amsart}
\usepackage{amssymb}
\usepackage{textcomp}
\usepackage[mathscr]{euscript}
\usepackage{pb-diagram}
\usepackage{enumitem}
\usepackage{comment}

\newtheorem{theorem}{Theorem}[section]
\newtheorem{qst}[theorem]{Question}
\newtheorem{thm}[theorem]{Theorem}
\newtheorem{prop}[theorem]{Proposition}

\newtheorem{fact}[theorem]{Fact}
\newtheorem{cor}[theorem]{Corollary}

\theoremstyle{definition}

\theoremstyle{remark}

\newcommand{\az}{\alpha \Z}
\newcommand{\fa}{f|_{[0,\alpha]}}

\ProvideTextCommandDefault{\cprime}{(U+042C)}

\newcommand{\nip}{\mathrm{NIP}}

\newcommand{\Cal}[1]{\ensuremath{\mathcal{#1}}}
\newcommand{\Sa}[1]{\ensuremath{\mathscr{#1}}}

\newcommand{\Z}{\mathbb{Z}}
\newcommand{\N}{\mathbb{N}}

\newcommand{\Q}{\mathbb{Q}}
\newcommand{\R}{\mathbb{R}}

\begin{document}
\title[]{Generalizing a theorem of B\`{e}s and Choffrut}

\author{Erik Walsberg}
\address{Department of Mathematics, Statistics, and Computer Science\\
Department of Mathematics\\University of California, Irvine, 340 Rowland Hall (Bldg.\# 400),
Irvine, CA 92697-3875}
\email{ewalsber@uci.edu}
\urladdr{http://www.math.illinois.edu/\textasciitilde erikw}

\date{\today}

\maketitle

\begin{abstract}
B\`{e}s and Choffrut recently showed that there are no intermediate structures between $(\R,<,+)$ and $(\R,<,+,\Z)$.
We prove a generalization: if $\Sa R$ is an o-minimal expansion of $(\R,<,+)$ by bounded subsets of Euclidean space then there are no intermediate structures between $\Sa R$ and $(\Sa R,\Z)$.
It follows there are no intermediate structures between $(\R,<,+,\sin|_{[0,2\pi]})$ and $(\R,<,+,\sin)$.
\end{abstract}

\section{Introduction}
\noindent
Throughout all structures are first order and ``definable" means ``first-order definable, possibly with parameters".
Suppose $\Sa M$, $\Sa N$, and $\Sa O$ are structures on a common domain $M$.
Then $\Sa M$ is a reduct of $\Sa O$ if every $\Sa M$-definable subset of every $M^n$ is $\Sa O$-definable, $\Sa M$ and $\Sa O$ are interdefinable if each is a reduct of the other, $\Sa M$ is a proper reduct of $\Sa O$ if $\Sa M$ is a reduct of $\Sa O$ and $\Sa M$ is not interdefinable with $\Sa O$, and $\Sa N$ is intermediate between $\Sa M$ and $\Sa O$ if $\Sa M$ is a proper reduct of $\Sa N$ and $\Sa N$ is a proper reduct of $\Sa O$.
If $A$ is a subset of $M^m$ then the structure induced on $A$ by $\Sa M$ is the structure on $A$ with an $n$-ary predicate defining $A^n \cap X$ for every $\Sa M$-definable $X \subseteq M^{mn}$.
\newline

\noindent
Throughout $\Sa R$ and $\Sa S$ are structures expanding $(\R,<,+)$.
Recall that $\Sa R$ is o-minimal if every definable subset of $\R$ is a finite union of open intervals and singletons and $\Sa R$ is locally o-minimal if one of the following equivalent conditions holds.
\begin{enumerate}
\item If $X \subseteq \R$ is definable then for any $a \in \R$ there is an open interval $I$ containing $a$ such that $I \cap X$ is a finite union of open intervals and singletons.
\item If $I$ is a bounded interval then the structure induced on $I$ by $\Sa R$ is o-minimal.
\end{enumerate}
It is clear that $(2)$ implies $(1)$, the other direction follows by compactness of closed bounded intervals.
It follows from independent work of Miller~\cite{ivp} or Weispfenning~\cite{weis} that $(\R,<,+,\Z)$ is locally o-minimal, another example of a locally o-minimal non-o-minimal structure is $(\R,<,+,\sin)$, see \cite{TV-local}.
\newline

\noindent
B\`{e}s and Choffrut~\cite{bes-choffrut} show that there is no structure intermediate between $(\R,<,+)$ and $(\R,<,+,\Z)$.
We prove the following.

\begin{thm}
\label{thm:main}
Suppose $\Sa R$ is o-minimal and $\alpha >0$.
If $(\Sa R,\az)$ is locally o-minimal then there is no structure intermediate between $\Sa R$ and $(\Sa R, \az)$.
Furthermore $(\Sa R,\az)$ is locally o-minimal if and only if $\Sa R$ has no poles and rational global scalars.
\end{thm}

\noindent 
As $(\R,<,+,\sin|_{[0,2\pi]})$ is o-minimal and $(\R,<,+,\sin|_{[0,2\pi]}, 2\pi\Z)$ is interdefinable with $(\R,<,+,\sin)$ there is no intermediate structure between $(\R,<,+,\sin|_{[0,2\pi]})$ and $(\R,<,+,\sin)$.
A pole is a definable surjection $\gamma : I \to J$ where $I$ is a bounded interval and $J$ is an unbounded interval and $\Sa R$ has rational global scalars if the function $\R \to \R$, $t \mapsto \lambda t$ is only definable when $\lambda \in \Q$.
\newline

\noindent
The second claim of Theorem~\ref{thm:main} is essentially proven in \cite{big-nip}.
Indeed, much of the proof of Theorem~\ref{thm:main} essentially appears in \cite{big-nip}.
As that paper is rather long and involves many abstract model-theoretic concepts, and Theorem~\ref{thm:main} is of independent interest, it seems worthwhile to record a separate proof.
For those interested in $\nip$ we note that the following are shown to be equivalent in \cite{big-nip}:
\begin{enumerate}
\item The primitive relations of $\Sa S$ are boolean combinations of closed sets, the primitive functions of $\Sa S$ are continuous, and $\Sa S$ is strongly dependent,
\item $\Sa S$ is either o-minimal or interdefinable with $(\Sa R,\az)$ for some real number $\alpha > 0$ and o-minimal $\Sa R$ with no poles and rational global scalars.
\end{enumerate}

\noindent
In contrast arbitrary strongly dependent expansions of $(\R,<,+)$ are as complicated as arbitrary strongly dependent structures of cardinality continuum by~\cite{HNW}.
\newline

\noindent
The theorem of B\`{e}s and Choffrut is analogous to the theorem of Conant~\cite{conant} that there are no intermediate structures between $(\Z,+)$ and $(\Z,<,+)$.
Conant's proof relies on a detailed analysis of $(\Z,<,+)$-definable subsets of $\Z^n$.
Alouf and d'Elb\'{e}e~\cite{AldE} give a quicker proof.
They use deep model-theoretic machinery to reduce to the unary case and thereby avoid geometric complexity.
They apply the earlier theorem of Conant and Pillay~\cite{CoPi} that there are no proper stable dp-minimal expansions of $(\Z,+)$, which itself depends on work of Palac\'{i}n and Sklinos~\cite{PS-superstable}, who apply the Buechler dichotomy theorem, a high-level result of stability theory.
B\`{e}s and Choffrut's proof involves a detailed geometric analysis of $(\R,<,+,\Z)$-definable subsets of $\R^n$.
Our proof relies on an important o-minimal structure theorem of Edmundo~\cite{ed-str} and Kawakami, Takeuchi, Tanaka, and Tsuboi's~\cite{KTTT} work on locally o-minimal structures.
These theorems reduce the first claim of Theorem~\ref{thm:main} to the unary case and also show that the unary case follows from the work of B\`{e}s and Choffrut.
\newline

\noindent
In Section~\ref{section:survey} we describe what we know about $(\Sa R,\Z)$ when $\Sa R$ is o-minimal.

\subsection*{Conventions}
$m,n$ are natural numbers and $\alpha,\beta,\lambda,s,t,r$ are real numbers.

\section{Proof of Theorem~\ref{thm:main}}
\label{section:short}
\noindent 
Fact~\ref{fact:ed} is a special case of a theorem of Edmundo~\cite{ed-str}.

\begin{fact}
\label{fact:ed}
Suppose $\Sa S$ is o-minimal.
Then the following are equivalent
\begin{enumerate}
\item $\Sa S$ has no poles,
\item $\Sa S$ does not define $\oplus,\otimes : \R^2 \to \R$ such that $(\R,<,\oplus,\otimes)$ is isomorphic to $(\R,<,+,\cdot)$,
\item there is a collection $\Cal B$ of bounded subsets of Euclidean space and a subfield $K$ of $(\R,+,\cdot)$ such that $\Sa S$ is interdefinable with $(\R,<,+,\Cal B, (t \mapsto \lambda t)_{\lambda \in K})$.
\end{enumerate}
Also $\Sa S$ has no poles and rational global scalars if and only if there is a collection $\Cal B$ of bounded subsets of Euclidean space such that $\Sa S$ is interdefinable with $(\R,<,+,\Cal B)$.
\end{fact}

\noindent
Let $\tilde{+}$ be the function $[0,1)^2 \to [0,1)$ where $s \tilde{+} t = s + t$ if $s + t < 1$ and $s \tilde{+} t = s + t - 1$ otherwise.
We will need two theorems of Kawakami, Takeuchi, Tanaka, and Tsuboi~\cite{KTTT} which in particular show that any locally o-minimal expansion of $(\R,<,+)$ which defines $\Z$ is bi-interpretable with the disjoint union of an o-minimal expansion of $([0,1),<,\tilde{+})$ and an arbitrary expansion of $(\Z,<,+)$.
A special case is that $(\R,<,+,\Z)$ is bi-interpretable with the disjoint union of $([0,1),<,\tilde{+})$ and $(\Z,<,+)$.
Fact~\ref{fact:kttt-1} is a special case of \cite[Theorem 18]{KTTT}.

\begin{fact}
\label{fact:kttt-1}
Suppose $\Sa I$ is an o-minimal expansion of $([0,1),<,\tilde{+})$ and $\Sa Z$ is an expansion of $(\Z,<,+)$.
Then there is an expansion $\Sa S$ of $(\R,<,+)$ such that a subset of $\R^n$ is $\Sa S$-definable if and only if it is a finite union of sets of the form $\bigcup_{b \in Y} b + X$ for $\Sa I$-definable $X \subseteq [0,1)^n$ and $\Sa Z$-definable $Y \subseteq \Z^n$.
This $\Sa S$ is locally o-minimal.
\end{fact}

\noindent
Note that local o-minimality of $\Sa S$ follows directly from the description of $\Sa S$-definable sets.
Fact~\ref{fact:kttt-2} is \cite[Theorem 24]{KTTT}.

\begin{fact}
\label{fact:kttt-2}
Suppose $\Sa S$ is locally o-minimal and defines $\Z$.
Then every $\Sa S$-definable subset of every $\R^n$ is a finite union of sets of the form $\bigcup_{b \in Y} b + X$ where $X \subseteq [0,1)^n$ and $Y \subseteq \Z^n$ are $\Sa S$-definable.
\end{fact}

\noindent
Proposition~\ref{prop:local-1} yields the second claim of Theorem~\ref{thm:main}.

\begin{prop}
\label{prop:local-1}
Suppose $\Sa R$ is o-minimal.
If $(\Sa R,\Z)$ is locally o-minimal then $\Sa R$ has no poles and rational global scalars.
If $\Sa R$ has no poles and rational global scalars then $(\Sa R,\Z)$ is locally o-minimal and furthermore every $(\Sa R,\Z)$-definable subset of every $\R^n$ is a finite union of sets of the form $\bigcup_{b \in Y} b + X$ for $\Sa R$-definable $X \subseteq [0,1)^n$ and $(\Z,<,+)$-definable $Y \subseteq \Z^n$.
\end{prop}

\begin{proof}
Suppose $\Sa R$ has a pole $\gamma : I \to J$.
Applying the o-minimal monotonicity theorem \cite[3.1.2]{lou-book} and shrinking $I$ and $J$ if necessary we suppose that $\gamma$ is continuous and strictly increasing or strictly decreasing.
After possibly reflecting and translating we suppose that $\R_{>0} \subseteq J$ and $\gamma$ is strictly increasing.
Then $\gamma^{-1}(\N)$ is an infinite bounded discrete subset of $\R$ so $(\Sa R,\Z)$ is not locally o-minimal.
Suppose $\lambda \in \R \setminus \Q$ is such that the function $\R \to \R$ given by $t \mapsto \lambda t$ is $\Sa R$-definable.
Then $\Z + \lambda\Z$ is dense and co-dense, so $(\Sa R,\Z)$ is not locally o-minimal.
\newline

\noindent
Now suppose that $\Sa R$ has no poles and rational global scalars.
Let $\Sa I$ be the structure induced on $[0,1)$ by $\Sa  R$.
Let $\Sa S$ be constructed from $\Sa I$ and $(\Z,<,+)$ as in the statement of Fact~\ref{fact:kttt-1}.
It is enough to show that $\Sa S$ and $(\Sa R,\Z)$ are interdefinable.
Every $\Sa I$-definable subset of $[0,1)^n$ is trivially $\Sa R$-definable, so it follows from Fact~\ref{fact:kttt-1} that $\Sa S$ is a reduct of $(\Sa R, \Z)$.
We show that $(\Sa R,\Z)$ is a reduct of $\Sa S$.
As $\Sa S$ defines $\Z$ it suffices to show that $\Sa R$ is a reduct of $\Sa S$.
Applying Fact~\ref{fact:ed} let $\Cal B$ be a collection of bounded $\Sa R$-definable sets such that $\Sa R$ and $(\R,<,+,\Cal B)$ are interdefinable.
Rescaling and translating, we may assume that each $X \in \Cal B$ is a subset of some $[0,1)^n$.
So every $X \in \Cal B$ is $\Sa I$-definable, hence $\Sa S$-definable.
So $(\R,<,+,\Cal B)$ is a reduct of $\Sa S$.
\end{proof}

\noindent
If $\Sa M$ is a structure with domain $M$ and $\Sa N$ is an expansion of $\Sa M$, then we say that $\Sa N$ is $\Sa M$-minimal if every $\Sa N$-definable subset of $M$ is $\Sa M$-definable.
($\Sa R$ is o-minimal if and only if $\Sa R$ is $(\R,<)$-minimal.)
Corollary~\ref{cor:minimal} follows from Proposition~\ref{prop:local-1}.

\begin{cor}
\label{cor:minimal}
Suppose $\Sa R$ is o-minimal and $(\Sa R,\Z)$ is locally o-minimal.
Then $(\Sa R,\Z)$ is $(\R,<,+,\Z)$-minimal.
\end{cor}

\newpage

\noindent
Theorem~\ref{thm:main-1} finishes the proof of Theorem~\ref{thm:main}.

\begin{thm}
\label{thm:main-1}
Suppose $\Sa R$ is o-minimal, $\alpha > 0$, and $(\Sa R,\az)$ is locally o-minimal.
Then there are no intermediate structures between $\Sa R$ and $(\Sa R,\az)$.
\end{thm}

\begin{proof}
Rescaling reduces to the case when $\alpha = 1$.
Suppose $\Sa S$ is a reduct of $(\Sa R,\Z)$ and $\Sa R$ is a reduct of $\Sa S$.
\newline

\noindent
Suppose $\Sa S$ is o-minimal.
As $(\Sa S,\Z)$ is locally o-minimal Proposition~\ref{prop:local-1} shows that $\Sa S$ has no poles and rational global scalars.
Applying Fact~\ref{fact:ed} let $\Cal B$ be a collection of bounded $\Sa S$-definable sets such that $\Sa S$ and $(\R,<,+,\Cal B)$ are interdefinable.
It follows from Proposition~\ref{prop:local-1} that every bounded $(\Sa R,\Z)$-definable set is $\Sa R$-definable.
So $\Sa S$ is interdefinable with $\Sa R$.
\newline

\noindent
Suppose $\Sa S$ is not o-minimal.
Then $\Sa S$ defines a subset $X$ of $\R$ which is not $\Sa R$-definable.
By Corollary~\ref{cor:minimal} $X$ is definable in $(\R,<,+,\Z)$.
So $(\R,<,+,X)$ defines $\Z$ by B\`{e}s and Choffrut~\cite{bes-choffrut}.
Thus $\Sa S$ and $(\Sa R,\Z)$ are interdefinable.
\end{proof}

\noindent
We prove two corollaries.

\begin{cor}
\label{cor:cob}
Suppose $\Sa R$ is o-minimal and has no poles and rational global scalars.
Suppose $\alpha,\beta > 0$ and $\alpha/\beta \notin \Q$.
Suppose $X \subseteq \R^n$ is definable in both $(\Sa R,\az)$ and $(\Sa R,\beta\Z)$.
Then $X$ is definable in $\Sa R$.
\end{cor}

\noindent
Note that if $\alpha/\beta \in \Q$ then $(\Sa R,\az)$ and $(\Sa R,\beta\Z)$ are interdefinable.

\begin{proof}
Suppose $X$ is not $\Sa R$-definable.
By Theorem~\ref{thm:main-1} $(\Sa R,X)$ defines $\az,\beta\Z$.
As $\az + \beta\Z$ is dense and co-dense in $\R$, $(\Sa R,X)$ is not locally o-minimal, contradiction.
\end{proof}

\noindent
A function $f : \R \to \R^n$ is periodic with period $\alpha \neq 0$ if $f(t + \alpha) = f(t)$ for all $t$.

\begin{cor}
\label{cor:sin}
Suppose $f : \R \to \R^n$ is analytic and  periodic with period $\alpha > 0$.
Then there is no structure intermediate between $(\R,<,+,f|_{[0,\alpha]})$ and $(\R,<,+,f)$.
\end{cor}

\noindent
So there is no structure intermediate between $(\R,<,+,\sin|_{[0,2\pi]})$ and $(\R,<,+,\sin)$.

\begin{proof}
If $f$ is constant then $(\R,<,+,f)$, $(\R,<,+,f|_{[0,\alpha]})$, and $(\R,<,+)$ are all interdefinable.
Suppose $f$ is non-constant.
\newline

\noindent
We show that $(\R,<,+,\fa,\az)$ and $(\R,<,+,f)$ are interdefinable.
We have
$$ f(t) = f( t - \max\{ a \in \az : a < t \}) \quad \text{for all  } t.$$
It follows that $f$ is $(\R,<,\fa,\az)$-definable.
To show that $(\R,<,+,\fa,\az)$ is a reduct of $(\R,<,+,f)$ it suffices to show that $\az$ is $(\R,<,+,f)$-definable.
Let $P$ be the set of $r \neq 0$ such that $f$ is $r$-periodic.
Then $P$ is an $(\R,<,+,f)$-definable subgroup of $(\R,+)$.
As $f$ is continuous and non-constant $P$ is not dense in $\R$, so $P = \lambda\Z$ for some $\lambda > 0$.
Then $\alpha \in \lambda\Z$, it follows that $\az$ is $(\R,<,+,\lambda\Z)$-definable.
\newline

\noindent
It is well-known that Gabrielov's complement theorem implies $(\R,<,+,g|_I)$ is o-minimal for any analytic $g : \R \to \R^n$ and bounded interval $I$.
So $(\R,<,+,f|_{[0,\alpha]})$ is o-minimal.
Applying Proposition~\ref{prop:local-1} and rescaling we see that $(\R,<,+,f|_{[0,\alpha]},\az)$ is locally o-minimal.
Now apply Theorem~\ref{thm:main-1}.
\end{proof}

\section{Expansions of o-minimal structures by $\Z$}
\label{section:survey}
\noindent
We survey structures of the form $(\Sa R,\Z)$ where $\Sa R$ is o-minimal.
This will require fundamental classification results on o-minimal structures and key results from the theory of general expansions of $(\R,<,+)$.
This class of structures contains, sometimes in a somewhat disguised form, many interesting structures.
\newline

\noindent
We first describe two opposing examples.
It is well-known that $(\R,<,+,\cdot,\Z)$ defines all Borel subsets of all $\R^n$.
So $(\R,<,+,\cdot,\Z)$ is totally wild from the model-theoretic viewpoint.
Fix $\lambda > 0$ and define $\lambda^\Z := \{ \lambda^m : m \in \Z \}$.
It follows from work of van den Dries~\cite{vdd-Powers2} that $(\R,<,+,\cdot,\lambda^\Z)$ admits quantifier elimination in a natural expanded language.
It follows that this structure is $\nip$ and definable sets are geometrically tame.
Let $\Sa I$ be the structure induced on $\R_{>0}$ by $(\R,<,+,\cdot)$ and $\Sa R$ be the pushforward of $\Sa I$ by $\log_\lambda : \R_{>0} \to \R$.
Then $\Sa R$ is an o-minimal expansion of $(\R,<,+)$ and $(\Sa R, \Z)$ is isomorphic to the structure induced on $\R_{>0}$ by $(\R,<,+,\cdot,\lambda^\Z)$.
So in this case $(\Sa R,\Z)$ is tame and not locally o-minimal.
\newline

\noindent
We gather some background results.
We say that $\Sa S$ is field-type if there is a bounded interval $I$ and $\Sa S$-definable $\oplus,\otimes : I^2 \to I$ such that $(I,<,\oplus,\otimes)$ and $(\R,<,+,\cdot)$ are isomorphic.
Fact~\ref{fact:ps-tri} is a part of the Peterzil-Starchenko trichotomy theorem~\cite{PS-Tri}.

\begin{fact}
\label{fact:ps-tri}
Suppose $\Sa S$ is o-minimal.
Then exactly one of the following holds.
\begin{enumerate}
\item $\Sa S$ is field-type.
\item $\Sa S$ is a reduct of $(\R,<,+,(t \mapsto \lambda t)_{\lambda \in \R })$.
\end{enumerate}
\end{fact}

\noindent
We will apply the classification of structures intermediate between $(\R,<,+)$ and $(\R,<,+,(t \mapsto \lambda t)_{\lambda \in \R})$.
This is probably not original.
We associate two subfields of $(\R,+,\cdot)$ to $\Sa S$.
The \textit{field of global scalars} is the set of $\lambda$ such that the function $\R \to \R$, $t \mapsto \lambda t$ is $\Sa S$-definable.
The \textit{field of local scalars} is the set of $\lambda$ for which there is $s > 0$ such that the function $[0,s) \to \R$, $t \mapsto \lambda t$ is $\Sa S$-definable.
It is easy to see that both of these sets are subfields.
It is also easy to see that the field of local scalars is the set of $\lambda$ such that $[0,s) \to \R$, $t \mapsto \lambda t$ is $\Sa S$-definable for any $s > 0$.
Given subfields $K \subseteq L$ of $(\R,+,\cdot)$ we let $\Sa V_{K,L}$ be the expansion of $(\R,<,+)$ by the function $\R \to \R$, $t \mapsto \lambda t$ for all $\lambda \in L$ and by the function $[0,1) \to \R$, $t \mapsto \lambda t$ for all $\lambda \in K$.
We let $\Sa V_{K} := (\R,<,+,(t \mapsto \lambda t)_{\lambda \in K})$ so $\Sa V_{K}$ and $\Sa V_{K,K}$ are interdefinable.

\begin{prop}
\label{prop:classify}
Suppose $\Sa S$ is a reduct of $(\R,<,+,(t \mapsto \lambda t)_{\lambda \in \R})$.
Let $K$ be the field of local scalars of $\Sa S$ and $L$ be the field of global scalars of $\Sa S$.
Then $\Sa S$ is interdefinable with $\Sa V_{K,L}$.
\end{prop}

\noindent
Our proof is rather sketchy.

\begin{proof}
Observe that $\Sa V_{K,L}$ is a reduct of $\Sa S$.
We show that $\Sa S$ is a reduct of $\Sa V_{K,L}$.
Applying Fact~\ref{fact:ed} we obtain a collection $\Cal B$ of bounded $(\R,<,+,(t \mapsto \lambda t)_{\lambda \in \R})$-definable sets such that $(\R,<,+,\Cal B, (t \mapsto \lambda t)_{\lambda \in L})$ is interdefinable with $\Sa S$.
Rescaling and translating we may suppose that every $X \in \Cal B$ is a subset of some $[0,1)^n$.
A straightforward application of the semilinear cell decomposition \cite[Corollary 7.6]{Lou} shows that every $X \in \Cal B$ is definable in the expansion of $(\R,<,+)$ by all functions $[0,1) \to \R$, $t \mapsto 
\lambda t$ for $\lambda \in K$.
So $\Sa S$ is a reduct of $\Sa V_{K,L}$.
\end{proof}

\newpage
\noindent
We also need the following theorem of Hieronymi~\cite{discrete}.
Fact~\ref{fact:hier} is not exactly stated in \cite{discrete} but follows immediately from the main theorem of that paper.

\begin{fact}
\label{fact:hier}
Suppose $\Sa S$ is field-type and defines a discrete subset $D$ of $\R^n$ and a function $f : D \to \R$ such that $f(D)$ is somewhere dense.
Then $\Sa S$ defines all bounded Borel subsets of all $\R^n$.
\end{fact}

\noindent
Fact~\ref{fact:hb} is due to Hieronymi and Balderrama~\cite{BH-Cantor}.
It is a corollary to the fundamental Hieronymi-Tychonievich theorem~\cite{HT}.

\begin{fact}
\label{fact:hb}
Suppose $D,E$ are $\Sa S$-definable discrete subsets of $\R^n$, $f : D \to \R$ and $g : E \to \R$ are $\Sa S$-definable, $f(D)$ and $g(E)$ are both somewhere dense, and 
$$ [ f(D) - f(D) ] \cap [ g(E) - g(E) ] = \{0\}.$$
Then $\Sa S$ defines all bounded Borel subsets of all $\R^n$.
\end{fact}

\subsection{Pole}
We first consider the case when $\Sa R$ has a pole.
Applying Fact~\ref{fact:ed} fix $\Sa R$-definable $\oplus,\otimes : \R^2 \to \R$ such that $(\R,<,\oplus,\otimes)$ is isomorphic to $(\R,<,+,\cdot)$.
Note that there is a unique isomorphism $\iota : (\R,<,\oplus,\otimes) \to (\R,<,+,\cdot)$.
Let $\Sa R'$ be the pushforward of $\Sa R$ by $\iota$, $\boxplus : \R^2 \to \R$ be the pushforward of $+$ by $\iota$, and $Z := \iota(\Z)$.
Then $\Sa R'$ is an o-minimal expansion of $(\R,<,+,\cdot)$ and $\iota$ gives an isomorphism $(\Sa R,\Z) \to (\Sa R',Z)$.
We now recall the \textit{o-minimal two group question}, see \cite{two-group}.

\begin{qst}
\label{qst:two}
Suppose $\Sa S$ is an o-minimal expansion of $(\R,<,+,\cdot)$ and $\odot : \R^2 \to \R$ is $\Sa S$-definable such that $(\R,<,\odot)$ is an ordered abelian group.
Must there be either an $\Sa S$-definable isomorphism $(\R,<,\odot) \to (\R,<,+)$ or $(\R,<,\odot) \to (\R_{>0},<,\cdot)$?
\end{qst}

\noindent
The o-minimal two group question is open.
However, it is shown in \cite{two-group} that if the Pfaffian closure of $\Sa S$ is exponentially bounded then the two group question has a positive answer for $\Sa S$.
Recall that an expansion of $(\R,<,+)$ is exponentially bounded if every definable function $\R \to \R$ is eventually bounded above by a compositional iterate of the exponential, the Pfaffian closure of an o-minimal expansion is o-minimal, and every known o-minimal expansion of $(\R,<,+,\cdot)$ is exponentially bounded.
So we assume that the two group question has a positive answer over $\Sa R'$.
\newline

\noindent
Suppose $\tau : (\R,<,\boxplus) \to (\R,<,+)$ is an $\Sa R'$-definable isomorphism.
Then $\tau \circ \iota$ is an isomorphism $(\R,<,+) \to (\R,<,+)$ so there is $\lambda > 0$ such that $(\tau \circ \iota)(t) = \lambda t$ for all $t$.
So $\tau(Z) = \lambda \Z$.
It follows that $(\Sa R',Z)$ defines $\Z$ and is therefore totally wild, hence $(\Sa R,\Z)$ is totally wild.
\newline

\noindent
Suppose $\tau : (\R,<,\boxplus) \to (\R_{>0},<,\cdot)$ is an $\Sa R'$-definable isomorphism.
Then $\tau \circ 
\iota$ is an isomorphism $(\R,<,+) \to (\R_{>0},<,\cdot)$ so there is $\lambda > 0$ such that $(\tau \circ \iota)(t) = \lambda^t$ for all $t.$
Then $\tau(Z) = \lambda^{\Z}$ and $(\Sa R' Z)$ is interdefinable with $(\Sa R', \lambda^\Z)$.
We say that $\Sa R'$ has \textit{rational exponents} if the function $\R_{>0} \to \R_{>0}$ given by $t \mapsto t^r$ is only $\Sa R'$-definable when $r \in \Q$.
If $r \in \R \setminus \Q$ then $\{ ab^r : a,b \in \lambda^\Z\}$ is dense in $\R_{>0}$.
So it follows from Fact~\ref{fact:hier} that if $\Sa R'$ has irrational exponents then $(\Sa R',\lambda^\Z)$ defines all Borel subsets of all $\R^n$, so in this case $(\Sa R,\Z)$ is totally wild.
Suppose $\Sa R'$ has rational exponents.
Generalizing \cite{vdd-Powers2} Miller and Speissegger have shown that in this case $(\Sa R',\lambda^\Z)$ admits quantifier elimination in a natural expanded language~\cite[Section 8.6]{Miller-tame}.
It follows that $(\Sa R',\lambda^\Z)$ is $\nip$ (see \cite{GH-Dependent}) and that $(\Sa R',\lambda^\Z)$-definable sets are geometrically tame (see \cite{Miller-tame,Tychon-thesis}).
So in this case $(\Sa R,\Z)$ is tame.

\subsection{No pole}
We now assume that $\Sa R$ has no poles.
If $\Sa R$ has rational global scalars then $(\Sa R,\Z)$ is locally o-minimal, we have a good description of $(\Sa R,\Z)$-definable sets by Proposition~\ref{prop:local-1}, and $(\Sa R,\Z)$ is strongly dependent (see \cite{big-nip}).
\newline

\noindent
Suppose $\Sa R$ has irrational global scalars and is field type.
Fix an irrational element $\lambda$ of the field of global scalars of $\Sa R$.
Then $\Z + \lambda\Z$ is dense in $\R$.
Fact~\ref{fact:hier} shows that $(\Sa R,\Z)$ defines all bounded Borel sets and is therefore totally wild.
\newline

\noindent
Suppose $\Sa R$ has irrational scalars and is not field type.
Applying Fact~\ref{fact:ps-tri} and Proposition~\ref{prop:classify} we may suppose that $\Sa R = \Sa V_{K,L}$ for subfields $K,L \neq \Q$.
If $L$ is not quadratic then it follows from a theorem of Hieronymi and Tychonievich~\cite[Theorem B]{HT} that $(\Sa V_L,\Z)$ defines all Borel sets.
\newline

\noindent
Suppose $L = \Q(\alpha)$ for a quadratic irrational $\alpha$.
Suppose $K$ is not $\Q(\alpha)$.
(This case is a slight extension of \cite{HT}.)
Fix positive $\beta \in K \setminus \Q(\alpha)$.
Let $f : \Z^2 \to \R$ be given by $f(k,k') = k + \alpha k'$.
Then $f$ is $(\Sa V_{K,L},\Z)$-definable and $f(\Z^2)$ is dense in $\R$.
Let $E$ be the set of $(k,k') \in \Z^2$ such that $0 \leq f(k,k') < 1$.
Let $g : E \to \R$ be given by $g(k,k') = \beta k + \beta\alpha k'$.
Then $E$ and $g$ are $(\Sa V_{K,L},\Z)$-definable and $g(E)$ is dense in $[0,\beta)$.
Observe that $f(\Z^2) - f(\Z^2) = f(\Z^2) = \Z + \alpha\Z$ and $g(E) - g(E) \subseteq \beta\Z + \beta\alpha\Z$.
As $\beta$ is not in $\Q(\alpha)$ elementary algebra yields $(\Z + \alpha\Z) \cap (\beta\Z + \beta\alpha\Z) = \{0\}$.
An application of Fact~\ref{fact:hb} shows that $(\Sa V_{K,L},\Z)$ defines all bounded Borel sets.
\newline

\noindent
One case remains, when $K = L$ is quadratic.
In this case something remarkable happens.
Consider $\Sa V_{K}$ for a quadratic subfield $K$ of $(\R,<,+,\cdot)$.
Let $\Sa B$ be the standard model $(\Cal P(\N),\N,\in,+1)$ of the monadic second order theory of one successor, i.e. we have a sort $\Cal P(\N)$ for the power set of $\N$, a sort for $\N$, the membership relationship $\in$ between these two sorts, and the successor function $+1$ on $\N$.
It is a theorem of B\"{u}chi~\cite{Buchi} that $\Sa B$ is decidable.
By \cite[Theorem D]{H-Twosubgroups} or \cite[Theorem C]{HW-Monadic} $(\Sa V_K,\Z)$ defines an isomorphic copy of $\Sa B$.
Hieronymi has shown that $\Sa B$ defines an isomorphic copy of $(\Sa V_K,\Z)$, see \cite[Theorem D]{H-multiplication} and \cite[Theorem C]{H-Twosubgroups}.
\newline

\noindent
If an expansion of $(\R,<,+)$ defines all bounded Borel sets then it defines an isomorphic copy of $(\R,<,+,\cdot,\Z)$.
It is also easy to see that $(\Sa R,\Z)$ is a reduct of $(\R,<,+,\cdot,\Z)$ for any o-minimal $\Sa R$.
So we see that if $\Sa R$ is a \textit{known} o-minimal expansion of $(\R,<,+)$ then $(\Sa R,\Z)$ is, up to mutual interpretation, one of the following:

\begin{enumerate}
\item $(\Sa R,\Z)$ where $\Sa R$ has no poles and rational scalars,
\item $(\Sa S,\lambda^\Z)$, $\Sa S$ an o-minimal expansion of $(\R,<,+,\cdot)$ with rational exponents,
\item $(\Cal P(\N), \N, \in, +1)$,
\item $(\R,<,+,\cdot,\Z)$.
\end{enumerate}
 

\bibliographystyle{abbrv}
\bibliography{NIP}
\end{document}